\documentclass[a4paper,12pt]{amsart}
\usepackage{amsmath,latexsym,amssymb,amsfonts}
\usepackage{amscd,amssymb,epsfig}

\usepackage[arrow, matrix, curve]{xy}

\usepackage{hyperref}
\usepackage{mathdots}

\numberwithin{equation}{section}

\newtheorem{theorem}{Theorem}[section]
\newtheorem{lemma}[theorem]{Lemma}
\newtheorem{corollary}[theorem]{Corollary}
\newtheorem{proposition}[theorem]{Proposition}

\theoremstyle{definition}
\newtheorem{definition}[theorem]{Definition}

\newtheorem{example}[theorem]{Example}
\newtheorem{remark}[theorem]{Remark}

\renewcommand{\epsilon}{\varepsilon}

\newcommand{\C}{\mathbb{C}}
\newcommand{\N}{\mathbb{N}}
\newcommand{\R}{\mathbb{R}}

\newcommand{\id}{\operatorname{id}}
\newcommand{\ev}{\operatorname{ev}}

\newcommand{\ran}{\operatorname{ran}}
\newcommand{\vN}{\operatorname{vN}}

\makeindex

\title[Absence of algebraic relations and of zero divisors]{Absence of algebraic relations and of zero divisors under the assumption of finite non-microstates free Fisher information}

\author[T. Mai]{Tobias Mai}
\address{Universit\"{a}t des Saarlandes, FR $6.1-$Mathematik, 66123 Saarbr\"{u}cken, Germany }
\email{mai@math.uni-sb.de}

\author[R. Speicher]{Roland Speicher}
\address{Universit\"{a}t des Saarlandes, FR $6.1-$Mathematik, 66123 Saarbr\"{u}cken, Germany }
\email{speicher@math.uni-sb.de}

\author[M. Weber]{Moritz Weber}
\address{Universit\"{a}t des Saarlandes, FR $6.1-$Mathematik, 66123 Saarbr\"{u}cken, Germany }
\email{weber@math.uni-sb.de}

\date{\today}

\begin{document}

\begin{abstract}
We show that in a tracial and finitely generated $W^\ast$-probability space existence of conjugate variables in an appropriate sense exclude algebraic relations for the generators.

Moreover, under the assumption of finite non-microstates free Fisher information, we prove that there are no zero divisors in the sense that the product of any non-commutative polynomial in the generators with any element from the von Neumann algebra is zero if and only if at least one of those factors is zero.
\end{abstract}

\maketitle

\section{Introduction}

In a groundbreaking series of papers \cite{Voi-Entropy-I, Voi-Entropy-II, Voi-Entropy-III, Voi-Entropy-IV, Voi-Entropy-V, Voi-Entropy-VI} (see also the survey \cite{Voi-Entropy-Surv}), Voiculescu transferred the notion of entropy and Fisher information to the world of noncommutative probability theory. One of the most striking results which came out of this program is certainly the proof of the fact that the free group factors do not possess Cartan subalgebras in \cite{Voi-Entropy-III}, which gives in particular the solution of the by then longstanding open question, whether every separable $\text{II}_1$-factor contains Cartan subalgebras.

We should note that Voiculescu gave in fact two different approaches to entropy and Fisher information in the non-commutative setting. The first one is based on the notion of matricial microstates and defines free entropy $\chi$ first and deduces free Fisher information $\Phi$ from $\chi$, the second is based on the notion of conjugate variables with respect to certain non-commutative derivatives and defines free Fisher information $\Phi^\ast$ first and then deduces free entropy $\chi^\ast$ from this quantity. We want to note that both constructions lead independently to objects which are, compared to the classical theory, justifiably called entropy and Fisher information, but it is still not known whether they coincide, respectively.

In the classical case, as well as in the one-variable free case, finiteness of entropy or of Fisher information imply some regularity of the corresponding distribution of the variables; in particular, they have a density (with respect to Lebesgue measure). In the non-commutative situation, the notion of a density does not make any direct sense, but still it is believed that the existence of finite free entropy or finite free Fisher information (in any of the two approaches) should correspond to some regularity property of the considered non-commutative distributions. Thus one expects many ''smooth" properties for random variables $X_1,\dots,X_n$ for which either one of the quantities $\chi(X_1,\dots,X_n)$, 
$\chi^*(X_1,\dots,X_n)$, $\Phi(X_1,\dots,X_n)$, or $\Phi^*(X_1,\dots,X_n)$ is finite. In particular, it is commonly expected that such a finiteness excludes non-trivial algebraic relations between the considered random variables. Up to now there is no proof of such a general statement. We will show here such
a result in the case of finite non-microstates free Fisher information $\Phi^*(X_1\dots,X_n)$. Actually, we will not only prove the absence of algebraic relations between the $X_1,\dots,X_n$, but also show that such cannot hold locally on non-trivial Hilbert subspaces; more formally we show that there are no non-zero divisors in the affiliated von Neumann algebra.

In the second named approach, the so-called non-microstates approach, Voiculescu's definition of conjugate variables in \cite{Voi-Entropy-V} depends fundamentally on the existence of non-commutative derivatives. But this makes it necessary to assume right from the beginning that there are no algebraic relations between the generators of the underlying von Neumann algebra. Although this does not cause any serious trouble, this assumption might seem a bit unnatural and it could be -- depending on the particular situation -- quite hard to check.

The first part of our article is devoted to a slightly different approach to conjugate variables. We will give a definition, which will circumvent the initial assumption of absence of algebraic relations. In fact, we will show that absence of algebraic relations is actually a consequence of the existence of conjugate variables. Thus, a posteriori, it will turn out that our notion of conjugate variables actually agrees with Voiculescu's definition. Nevertheless, the assumption of algebraic freeness becomes redundant and the defining conditions are much easier to handle than those for the algebraic freeness of the generators.

Inspired by the observation that the existence of conjugate variables excludes algebraic relations, the second part of our article is concerned with similar but more advanced consequences of the finiteness of non-microstates Fisher information. More precisely, we will show that it excludes zero divisors in the sense that the product of any non-commutative polynomial in the generators with any element from the von Neumann algebra is zero if and only if at least one of those factors is zero.

In particular, this result allows to conclude that the distribution of any (of course, non-trivial) self-adjoint non-commutative polynomial in the generators does not have atoms, if the generators have finite free Fisher information $\Phi^*$. This extends the previous work of Shlyakhtenko and Skoufranis \cite{Shlyakhtenko-Skoufranis}.

\section{Existence of conjugate variables and absence of algebraic relations}\label{algebraic_relations}

Let $\C\langle Z_1,\dots,Z_n\rangle$ be the $\ast$-algebra of non-commutative polynomials in $n$ self-adjoint (formal) variables $Z_1,\dots,Z_n$. For $j=1,\dots,n$, we denote by $\partial_j$ the non-commutative derivative with respect to $Z_j$, i.e. $\partial_j$ is the unique derivation
$$\partial_j:\ \C\langle Z_1,\dots,Z_n\rangle \to \C\langle Z_1,\dots,Z_n\rangle \otimes \C\langle Z_1,\dots,Z_n\rangle$$
that satisfies $\partial_j Z_i = \delta_{i,j} 1 \otimes 1$ for $i=1,\dots,n$. Recall that being a derivation means for $\partial_j$ that
\begin{equation}\label{derivation-def}
\partial_j (P_1 P_2) = (\partial_j P_1) (1 \otimes P_2) + (P_1 \otimes 1) (\partial_j P_2)
\end{equation}
holds for all $P_1,P_2\in \C\langle Z_1,\dots,Z_n\rangle$. More explicitly, $\partial_j$ acts on monomials $P$ as
$$\partial_j P = \sum_{P = P_1 Z_j P_2} P_1 \otimes P_2,$$
where the sum runs over all decompositions $P=P_1Z_jP_2$ of $P$ with monomials $P_1,P_2$.

Throughout the following, let $(M,\tau)$ be a tracial $W^\ast$-probability space, which means that $M$ is a von Neumann algebra and $\tau$ is a faithful normal tracial state on $M$. For selfadjoint $X_1,\dots,X_n\in M$ we 
denote by $\vN(X_1,\dots,X_n)\subset M$ the von Neumann subalgebra of $M$ which is generated by $X_1,\dots,X_n$ and by
$L^2(X_1,\dots,X_n,\tau)\subset L^2(M,\tau)$ the $L^2$-space which is generated by $X_1,\dots,X_n$ with respect to the inner product given by $\langle P,Q\rangle := \tau(PQ^\ast)$.

In the following, we will denote by $\ev_X$, for a given $n$-tuple $X=(X_1,\dots,X_n)$ of self-adjoint elements of $M$, the homomorphism
$$\ev_X:\ \C\langle Z_1,\dots,Z_n\rangle \rightarrow \C\langle X_1,\dots,X_n\rangle \subset M$$
given by evaluation at $X=(X_1,\dots,X_n)$, i.e. the homomorphism $\ev_X$ is determined by $Z_i \mapsto X_i$.

For reasons of clarity, we put $P(X) := \ev_X(P)$ for any $P\in \C\langle Z_1,\dots,Z_n\rangle$ and $Q(X) := (\ev_X \otimes \ev_X)(Q)$ for any $Q\in\C\langle Z_1,\dots,Z_n\rangle^{\otimes 2}$.

\begin{definition}\label{conjugate_system}
Let $X_1,\dots,X_n\in M$ be self-adjoint elements. If there are elements $\xi_1,\dots,\xi_n\in L^2(M,\tau)$, such that
\begin{equation}\label{conjugate_relations}
(\tau\otimes\tau)((\partial_jP)(X_1,\dots,X_n)) = \tau(\xi_j P(X_1,\dots,X_n))
\end{equation}
is satisfied for each non-commutative polynomial $P\in\C\langle Z_1,\dots,Z_n\rangle$ and for $j=1,\dots,n$, then we say that $(\xi_1,\dots,\xi_n)$ \emph{satisfies the conjugate relations for $(X_1,\dots,X_n)$}.

If, in addition, $\xi_1,\dots,\xi_n$ belong to $L^2(X_1,\dots,X_n,\tau)$, we say that $(\xi_1,\dots,\xi_n)$ is the \emph{conjugate system for $(X_1,\dots,X_n)$}. 
\end{definition}

Like in the usual setting, we have the following.

\begin{remark}
Let $\pi$ be the orthogonal projection from $L^2(M,\tau)$ to $L^2(X_1,\dots,X_n,\tau)$. It is easy to see that if $(\xi_1,\dots,\xi_n)$ satisfies the conjugate relations for $(X_1,\dots,X_n)$, then $(\pi(\xi_1),\dots,\pi(\xi_n))$ satisfies the conjugate relations for $(X_1,\dots,X_n)$ as well, and is therefore a conjugate system for $(X_1,\dots,X_n)$.

It is an easy consequence of its defining property \eqref{conjugate_relations} that a conjugate system $(\xi_1,\dots,\xi_n)$ for $(X_1,\dots,X_n)$ is unique if it exists. 
\end{remark}

Note that our notion of conjugate relations and conjugate variables differs from the usual definition which was given by Voiculescu in \cite{Voi-Entropy-V}, roughly speaking, just by the placement of brackets. To be more precise, in \eqref{conjugate_relations}, we first apply the derivative $\partial_j$ to the given non-commutative polynomial $P$ before we apply the evaluation at $X=(X_1,\dots,X_n)$, instead of applying the evaluation first, which consequently makes it necessary to have in the second step a well-defined derivation on $\C\langle X_1,\dots,X_n\rangle$ corresponding to $\partial_j$.

From a more abstract point of view, this idea is in the same spirit as \cite[Lemma 3.2]{Shlyakhtenko} but only on a purely algebraic level. In fact, we used the surjective homomorphism $\ev_X: \C\langle Z_1,\dots,Z_n\rangle \rightarrow \C\langle X_1,\dots,X_n\rangle$ in order to pass from $(\C\langle X_1,\dots,X_n\rangle,\tau)$ to the non-commutative probability space $(\C\langle Z_1,\dots,Z_n\rangle,\tau_X)$, where $\tau_X := \tau\circ\ev_X$. Due to this lifting, the algebraic relations between the generators disappear whereas the relevant information about their joint distribution remains unchanged.

Our aim is to show that the existence of a conjugate system guarantees that $X_1,\dots,X_n$ do not satisfy any algebraic relations. Obviously, we can rephrase this in more algebraic terms by saying that the two sided ideal
$$I^1_X := \{P \in \C\langle Z_1,\dots,Z_n\rangle|\ P(X_1,\dots,X_n) = 0\}$$
of $\C\langle Z_1,\dots,Z_n\rangle$ is the zero ideal. But this exactly means that the evaluation homomorphism $\ev_X$ is in fact an isomorphism. Hence, if this is shown, we can immediately define a non-commutative derivation
$$\hat{\partial}_j:\ \C\langle X_1,\dots,X_n\rangle \rightarrow \C\langle X_1,\dots,X_n\rangle \otimes \C\langle X_1,\dots,X_n\rangle,$$
where the terminology derivation has to be understood with respect to the $\C\langle X_1,\dots,X_n\rangle$-bimodule structure of $\C\langle X_1,\dots,X_n\rangle^{\otimes 2}$ like in \eqref{derivation-def}.

Surprisingly, it turns out that, in order to prove $I_X^1=\{0\}$, it is helpful to consider this question at once together with the question of the existence of well-defined derivations $\hat{\partial}_j$ on $\C\langle X_1,\dots,X_n\rangle$.

\begin{proposition}\label{derivation_existence_1}
As before, let $(M,\tau)$ be a tracial $W^\ast$-probability space and let self-adjoint $X_1,\dots,X_n\in M$ be given. Assume that there are elements $\xi_1,\dots,\xi_n\in L^2(M,\tau)$, such that $(\xi_1,\dots,\xi_n)$ satisfies the conjugate relations \eqref{conjugate_relations} for $(X_1,\dots,X_n)$.

Corresponding to $X=(X_1,\dots,X_n)$, we introduce the following two-sided ideals of $\C\langle Z_1,\dots,Z_n\rangle$ and of $\C\langle Z_1,\dots,Z_n\rangle^{\otimes 2}$, respectively: 
$$I_X^1 := \{P \in \C\langle Z_1,\dots,Z_n\rangle|\ P(X_1,\dots,X_n) = 0\}$$
and
$$I_X^2 := \{Q \in \C\langle Z_1,\dots,Z_n\rangle^{\otimes 2}|\ Q(X_1,\dots,X_n) = 0\}.$$
Then, for each $j=1,\dots,n$,
$$P + I_X^1 \mapsto \partial_j P + I_X^2$$
induces a well-defined derivation
$$\hat{\partial}_j:\ \C\langle Z_1,\dots,Z_n\rangle/I_X^1 \rightarrow \C\langle Z_1,\dots,Z_n\rangle^{\otimes 2}/I_X^2.$$
\end{proposition}

Before starting the proof, let us introduce a binary operation $\sharp$ on the algebraic tensor product $M \otimes M$ by bilinear extension of
$$(a_1 \otimes a_2) \sharp (b_1 \otimes b_2) := (a_1 b_1) \otimes (b_2 a_2).$$ 

\begin{proof}[Proof of Proposition \ref{derivation_existence_1}]
Obviously, it is sufficient to show that $P\in I_X^1$ implies $\partial_j P \in I_X^2$ for all $j=1,\dots,n$. For seeing this, let $P\in I_X^1$ be given. If we take arbitrary $P_1,P_2\in \C\langle Z_1,\dots,Z_n\rangle$, we have that
$$\partial_j (P_1 P P_2) = (\partial_j P_1) P P_2 + P_1 P (\partial_j P_2) + P_1 (\partial_j P) P_2$$
and therefore, since $P(X)=0$,
$$(\tau \otimes \tau)(\partial_j (P_1 P P_2) (X)) = (\tau \otimes \tau)(P_1(X) \partial_j P(X) P_2(X)).$$
Furthermore, according to \eqref{conjugate_relations}, we may deduce that
$$(\tau \otimes \tau)(\partial_j (P_1 P P_2) (X)) = \tau(\xi_j (P_1 P P_2)(X)) = 0.$$
Thus, we observe that
$$(\tau \otimes \tau)((P_1 \otimes P_2)(X) \sharp \partial_jP(X)) = (\tau \otimes \tau)(P_1(X) \partial_j P(X) P_2(X)) = 0$$
for all $P_1,P_2 \in \C\langle Z_1,\dots,Z_n\rangle$ and hence by linearity
$$(\tau \otimes \tau)(Q(X) \sharp \partial_jP(X)) = 0$$
for all $Q\in \C\langle Z_1,\dots,Z_n\rangle^{\otimes 2}$. If we apply this observation to $Q = (\partial_j P)^\ast$, we easily see that $\partial_j P(X) = 0$ (recall that $\tau$ was assumed to be faithful), which means that $\partial_j P \in I_X^2$. This shows that $\hat{\partial}_j$ is indeed well-defined.

Due to the obvious fact that $I_X^1$ is a two-sided ideal in $\C\langle Z_1,\dots,Z_n\rangle$ as well as $I_X^2$ is a two-sided $\ast$-ideal in $\C\langle Z_1,\dots,Z_n\rangle^{\otimes 2}$, we see that $\C\langle Z_1,\dots, Z_n\rangle/I_X^1$ is a $\ast$-algebra as well as $\C\langle Z_1,\dots,Z_n\rangle^{\otimes 2}/I_X^2$, where the multiplication and the involution are just defined via representatives. Using this, it is easy to check that $\hat{\partial}_1,\dots,\hat{\partial}_n$ are indeed derivations.
\end{proof}

Basic linear algebra shows that
$$\C\langle Z_1,\dots, Z_n\rangle/I_X^1 \cong \C\langle X_1,\dots,X_n\rangle$$
and
$$\C\langle Z_1,\dots,Z_n\rangle^{\otimes 2}/I_X^2 \cong \C\langle X_1,\dots,X_n\rangle^{\otimes 2}.$$
Hence, Proposition \ref{derivation_existence_1} immediately implies the following corollary.

\begin{corollary}\label{derivation_existence_2}
As before, let $(M,\tau)$ be a tracial $W^\ast$-probability space and let self-adjoint $X_1,\dots,X_n\in M$ be given. Assume that there are elements $\xi_1,\dots,\xi_n\in L^2(M,\tau)$, such that $(\xi_1,\dots,\xi_n)$ satisfies the conjugate relations \eqref{conjugate_relations} for $(X_1,\dots,X_n)$.

Then, for each $j=1,\dots,n$, there is a unique derivation
$$\hat{\partial}_j:\ \C\langle X_1,\dots,X_n\rangle \rightarrow \C\langle X_1,\dots,X_n\rangle \otimes \C\langle X_1,\dots,X_n\rangle$$
such that the following diagram commutes.
\begin{equation}\label{derivation_diagram}
\begin{xy}
 \xymatrix{\C\langle Z_1,\dots, Z_n\rangle \ar[rr]^{\partial_j} \ar@{->>}[dd]_{\ev_X} & & \C\langle Z_1,\dots,Z_n\rangle^{\otimes 2}\ar@{->>}[dd]^{\ev_X \otimes \ev_X}\\
          & & \\
           \C\langle X_1,\dots,X_n\rangle \ar[rr]^{\partial_j} \ar[rr]^{\hat{\partial}_j} & & \C\langle X_1,\dots,X_n\rangle^{\otimes 2}}
\end{xy}
\end{equation}
\end{corollary}

The fact that the diagram in \eqref{derivation_diagram} commutes immediately implies that the derivations $\hat{\partial}_j$, $j=1,\dots,n$, satisfy $\hat{\partial}_j(X_i) = \delta_{j,i} 1 \otimes 1$ for all $j,i=1,\dots,n$. Indeed,
\begin{align*}
\hat{\partial}_j(X_i) &= \hat{\partial}_j(\ev_X(Z_i))\\
                      &= (\ev_X \otimes \ev_X)(\partial_j Z_i)\\
                      &= (\ev_X \otimes \ev_X)(\delta_{j,i} 1 \otimes 1)\\
                      &= \delta_{j,i} 1 \otimes 1.
\end{align*}
For reasons of completeness, we want to mention that the converse is also true, so that both statements are in fact equivalent: If we assume that the derivations $\hat{\partial}_j$, $j=1,\dots,n$, satisfy $\hat{\partial}_j(X_i) = \delta_{j,i} 1 \otimes 1$ for all $j,i=1,\dots,n$, then
$$d_j := \hat{\partial}_j \circ \ev_X:\ \C\langle Z_1,\dots, Z_n\rangle \rightarrow \C\langle X_1,\dots,X_n\rangle^{\otimes 2}.$$
obviously defines a derivation, where we consider $\C\langle X_1,\dots,X_n\rangle^{\otimes 2}$ as $\C\langle Z_1,\dots,Z_n\rangle$-bimodule via evaluation $\ev_X$. Thus, we have for each $P\in\C\langle Z_1,\dots,Z_n\rangle$ and $j=1,\dots,n$ that
$$d_j(P) = \sum_{i=1}^n (\ev_X\otimes\ev_X)(\partial_i P) \sharp d_j(Z_i).$$
Since $d_j(Z_i) = \hat{\partial}_j(X_i) = \delta_{j,i} 1 \otimes 1$ holds for $i=1,\dots,n$ by assumption, we see that $d_j(P) = (\ev_X\otimes\ev_X)(\partial_j P)$. By definition of $d_j$, this gives
$$(\hat{\partial}_j\circ\ev_X)(P) = d_j(P) = ((\ev_X\otimes\ev_X) \circ \partial_j) (P),$$
which precisely means that the diagram in \eqref{derivation_diagram} commutes.

However, this allows us to rephrase Corollary \ref{derivation_existence_2} as follows.

\begin{corollary}\label{derivation_existence_3}
As before, let $(M,\tau)$ be a tracial $W^\ast$-probability space and let self-adjoint $X_1,\dots,X_n\in M$ be given. Assume that there are elements $\xi_1,\dots,\xi_n\in L^2(M,\tau)$, such that \eqref{conjugate_relations} is satisfied for $j=1,\dots,n$.

Then, for each $j=1,\dots,n$, there is a unique derivation
$$\hat{\partial}_j:\ \C\langle X_1,\dots,X_n\rangle \rightarrow \C\langle X_1,\dots,X_n\rangle \otimes \C\langle X_1,\dots,X_n\rangle,$$
such that $\hat{\partial}_j(X_i) = \delta_{j,i} 1 \otimes 1$ holds for all $i=1,\dots,n$.
\end{corollary}

This final conclusion shows that even without assuming algebraic freeness of the generators $X_1,\dots,X_n$, it is possible to define non-commutative derivations that behave exactly like if the generators would be algebraically free. In fact, this is the key observation to reach our desired result.

\begin{proposition}\label{derivations_relations}
Let $M$ be a tracial $W^\ast$-probability space and let self-adjoint elements $X_1,\dots,X_n\in M$ be given. Assume that there are derivations $$\hat{\partial}_j:\ \C\langle X_1,\dots,X_n\rangle \rightarrow \C\langle X_1,\dots,X_n\rangle \otimes \C\langle X_1,\dots,X_n\rangle$$
for $j=1,\dots,n$, such that $\hat{\partial}_j(X_i) = \delta_{j,i} 1 \otimes 1$ holds for all $i=1,\dots,n$. Then there is no non-zero polynomial $P\in\C\langle Z_1,\dots,Z_n\rangle$ such that $P(X_1,\dots,X_n)=0$ holds.
\end{proposition}

\begin{proof}
First of all, we observe that for any $P\in\C\langle Z_1,\dots,Z_n\rangle$ the following implication holds true:
$$P(X_1,\dots,X_n) = 0 \quad\Longrightarrow\quad \forall j=1,\dots,n:\ (\partial_j P)(X_1,\dots,X_n) = 0$$
This follows immediately from the assumption that the diagram \eqref{derivation_diagram} commutes. Thus, we may define $\Delta_j := ((\tau\circ\ev_X)\otimes\id) \circ \partial_j$, which is a linear mapping
$$\Delta_j:\ \C\langle Z_1,\dots,Z_n\rangle \rightarrow \C\langle Z_1,\dots,Z_n\rangle,$$
such that for any $P\in\C\langle Z_1,\dots,Z_n\rangle$ the implication
$$P(X_1,\dots,X_n) = 0 \quad\Longrightarrow\quad \forall j=1,\dots,n:\ (\Delta_j P)(X_1,\dots,X_n) = 0$$
holds true. Now, let us assume that there is a non-zero polynomial $P\in\C\langle Z_1,\dots,Z_n\rangle$ such that $P(X_1,\dots,X_n)=0$ holds, for instance
$$P(Z_1,\dots,Z_n) = a_0 + \sum_{k=1}^d \sum_{i_1,\dots,i_k = 1}^n a_{i_1,\dots,i_k} Z_{i_1} \dots Z_{i_k},$$
where $d\geq1$ denotes the total degree of $P$. We choose any summand of highest degree
$$a_{i_1,\dots,i_d} Z_{i_1} \dots Z_{i_d}$$
of $P$. It is easy to see that $\Delta_{i_d} \dots \Delta_{i_1} P = a_{i_1,\dots,i_d}$. Hence, we deduce
$$a_{i_1,\dots,i_n} = (\Delta_{i_d} \dots \Delta_{i_1} P)(X_1,\dots,X_n) = 0,$$
which finally leads to a contradiction. Therefore, we must have $P=0$, which concludes the proof.
\end{proof}

Combining the above Proposition \ref{derivations_relations} with Corollary \ref{derivation_existence_3} leads us directly to the following theorem.

\begin{theorem}\label{main1}
As before, let $(M,\tau)$ be a tracial $W^\ast$-probability space. Let $X_1,\dots,X_n\in M$ be self-adjoint and assume that there are elements $\xi_1,\dots,\xi_n\in L^2(M,\tau)$, such that $(\xi_1,\dots,\xi_n)$ satisfies the conjugate relations for $(X_1,\dots,X_n)$, i.e. \eqref{conjugate_relations} holds for $j=1,\dots,n$. Then we have the following statements:
\begin{itemize}
 \item[(a)] $X_1,\dots,X_n$ do not satisfy any non-trivial algebraic relation, i.e. there exists no non-zero polynomial $P\in\C\langle Z_1,\dots,Z_n\rangle$ such that $P(X_1,\dots,X_n)=0$.
 \item[(b)] For $j=1,\dots,n$, there is a unique derivation
$$\hat{\partial}_j:\ \C\langle X_1,\dots,X_n\rangle \rightarrow \C\langle X_1,\dots,X_n\rangle \otimes \C\langle X_1,\dots,X_n\rangle$$
which satisfies $\hat{\partial}_j(X_i) = \delta_{j,i} 1 \otimes 1$ for $i=1,\dots,n$.
\end{itemize}
\end{theorem}

Since Theorem \ref{main1} clarifies which consequences the existence of the conjugate system has, we may proceed now by defining (non-microstates) free Fisher information.

\begin{definition}
Let $(M,\tau)$ be a tracial $W^\ast$-probability space and let self-adjoint elements $X_1,\dots,X_n \in M$ be given. We define their \emph{(non-microstates) free Fisher information $\Phi^\ast(X_1,\dots,X_n)$} by
$$\Phi^\ast(X_,\dots,X_n) := \sum^n_{j=1} \|\xi_j\|_2^2$$
if a conjugate system $(\xi_1,\dots,\xi_n)$ for $(X_1,\dots,X_n)$ in the sense of Definition \ref{conjugate_system} exists, and we put $\Phi^\ast(X_1,\dots,X_n) := \infty$ if no such conjugate system for $(X_1,\dots,X_n)$ exists.
\end{definition}

The quantity $\Phi^\ast(X_1,\dots,X_n)$ is obviously well-defined, since as soon as a conjugate system in the sense of Definition \ref{conjugate_system} exists, Theorem \ref{main1} implies that it is also a conjugate system in the usual sense. Thus, $\Phi^\ast(X_1,\dots,X_n)$ is just the usual non-microstates free Fisher information which was defined in \cite{Voi-Entropy-V}.

Nevertheless, it has the advantage that it can be defined even without assuming the algebraic freeness of $X_1,\dots,X_n$ right from the beginning.

Let $(M,\tau)$ be a $W^\ast$-probability space and let self-adjoint elements $X_1,\dots,X_n\in M$ be given such that the condition $\Phi^\ast(X_1,\dots,X_n) < \infty$ is fulfilled. Theorem \ref{main1} tells us that, for $j=1,\dots,n$, there is a unique derivation
$$\hat{\partial}_j:\ \C\langle X_1,\dots,X_n\rangle \rightarrow \C\langle X_1,\dots,X_n\rangle \otimes \C\langle X_1,\dots,X_n\rangle$$
that satisfies $\hat{\partial}_j(X_i) = \delta_{j,i} 1 \otimes 1$ for $i=1,\dots,n$. But furthermore, it tells us that $X_1,\dots,X_n$ do not satisfy any algebraic relation, which in other words means that the evaluation homomorphism $\ev_X$ induces an isomorphism between the abstract polynomial algebra $\C\langle Z_1,\dots, Z_n\rangle$ and the subalgebra $\C\langle X_1,\dots,X_n\rangle$ of $M$. We have seen in Corollary \ref{derivation_existence_2} that under this identification the derivations $\partial_j$ on $\C\langle Z_1,\dots, Z_n\rangle$ and $\hat{\partial_j}$ on $\C\langle X_1,\dots,X_n\rangle$ correspond to each other. We will therefore not distinguish anymore between $\partial_j$ and $\hat{\partial}_j$.

We finish this section by noting that $\Phi^\ast(X_1,\dots,X_n)<\infty$ moreover excludes analytic relations. More precisely, this means that there is no non-zero non-commutative power series $P$, which is convergent on a polydisc
$$D_R:=\{(Y_1,\dots,Y_n)\in M^n|\ \forall j=1,\dots,n:\ \|Y_j\| < R\}$$
for some $R>0$, such that $(X_1,\dots,X_n)\in D_R$ and $P(X_1,\dots,X_n)=0$. Based on Voiculescu's original definition of the non-microstates free Fisher information and hence under the additional assumption that $X_1,\dots,X_n$ are algebraically free, this was shown by Dabrowski in \cite[Lemma 37]{Dab-free_stochastic_PDE}.

\section{Finite free Fisher information and zero divisors}

Inspired by the methods used in the proof of our Theorem \ref{main1}, we address now the more general question of existence of zero divisors under the assumption of finite non-microstates free Fisher information.

First, we shall make more precise, what we mean by this. Our aim is to prove the following statement.

\begin{theorem}\label{main2}
Let $(M,\tau)$ be a tracial $W^\ast$-probability space. Furthermore, let $X_1,\dots,X_n \in M$ be self-adjoint elements and assume that $\Phi^\ast(X_1,\dots,X_n) < \infty$ holds.

Then, for any non-zero non-commutative polynomial $P$, there exists no non-zero self-adjoint element $w\in\vN(X_1,\dots,X_n)$ such that
$$P(X_1,\dots,X_n) w = 0.$$
\end{theorem}

Recall that to each element $X=X^\ast\in M$, there corresponds a unique probability measure $\mu_X$ on the real line $\R$, which has the same moments as $X$, i.e. it satisfies 
$$\tau(X^k) = \int_\R t^k\, d\mu_X(t) \qquad\text{for $k=0,1,2,\dots$}.$$
It is an immediate consequence of Theorem \ref{main2} that the distribution $\mu_{P(X_1,\dots,X_n)}$ of $P(X_1,\dots,X_n)$ for any non-constant self-adjoint polynomial $P$ cannot have atoms, if $\Phi^\ast(X_1,\dots,X_n) < \infty$. The precise statement reads as follows. 

\begin{corollary}
Let $(M,\tau)$ be a tracial $W^\ast$-probability space and let $X_1,\dots,X_n \in M$ be self-adjoint with $\Phi^\ast(X_1,\dots,X_n) < \infty$. Then, for any non-constant self-adjoint non-commutative polynomial $P$, the distribution $\mu_{P(X_1,\dots,X_n)}$ of $P(X_1,\dots,X_n)$ does not have atoms.
\end{corollary}

\subsection{Collecting the ingredients for the proof of Theorem \ref{main2}}

Corresponding to the assumptions made in Theorem \ref{main2}, let $(M,\tau)$ be a tracial $W^\ast$-probability space and let $X_1,\dots,X_n \in M$ be self-adjoint with $\Phi^\ast(X_1,\dots,X_n) < \infty$.

As we have seen in Section \ref{algebraic_relations}, those conditions guarantee that, for $j=1,\dots,n$, there exists a unique derivation
$$\partial_j:\ \C\langle X_1,\dots,X_n\rangle \rightarrow \C\langle X_1,\dots,X_n\rangle \otimes \C\langle X_1,\dots,X_n\rangle,$$
which is determined by the condition $\partial_j X_i = \delta_{i,j} 1 \otimes 1$ for $i=1,\dots,n$. For each $j=1,\dots,n$, we may consider $\partial_j$ as a densely defined unbounded operator
$$\partial_j:\ L^2(M,\tau) \supseteq D(\partial_j) \rightarrow L^2(M \overline{\otimes} M, \tau \otimes \tau)$$
with domain $D(\partial_j) := \C\langle X_1,\dots,X_n\rangle$, where we denote by $M \overline{\otimes} M$ the von Neumann algebra tensor product of $M$ with itself.

Since due to $\Phi^\ast(X_1,\dots,X_n) < \infty$ a conjugate system $(\xi_1,\dots,\xi_n)$ for $(X_1,\dots,X_n)$ exists, we see by \eqref{conjugate_relations} that $1 \otimes 1$ belongs to the domain of definition of the adjoints $\partial_1^\ast,\dots,\partial_n^\ast$ and that $\xi_j = \partial_j^\ast(1 \otimes 1)$ holds for $j=1,\dots,n$.

The proof of Theorem \ref{main2} will be based on several well-knowns facts about those operators $\partial_j$, which we collect here for reader's convenience.

\begin{lemma}[Corollary 4.2 and Proposition 4.3 in \cite{Voi-Entropy-V}]\label{Voi-formula}
Under the above conditions, we have for $j=1,\dots,n$ that
$$\C\langle X_1,\dots,X_n\rangle \otimes \C\langle X_1,\dots,X_n\rangle \subseteq D(\partial_j^\ast),$$
i.e. $\partial_j^\ast$ is densely defined as well and $\partial_j$ is closable. More explicitly, we have for each $Y\in \C\langle X_1,\dots,X_n\rangle^{\otimes 2}$ the formula
$$\partial_j^\ast(Y) = m_{\xi_j}(Y) - m_1(\id\otimes\tau\otimes\id)(\partial_j\otimes\id+\id\otimes\partial_j)(Y).$$
\end{lemma}

Here, for any $\eta\in L^2(M,\tau)$, we denote by $m_\eta$ the linear operator $m_\eta: M\otimes M \rightarrow L^2(M,\tau)$ defined on the algebraic tensor product $M\otimes M$, which is given by $m_\eta(a_1 \otimes a_2) := a_1\eta a_2$. And thus of course, $m_1(a_1\otimes a_2)=a_1a_2$.\

The lemma above allows us to conclude that in particular
\begin{equation}\label{conjugate_formulas}
\begin{aligned}
\partial_j^\ast(P\otimes 1) &= P \xi_j - (\id\otimes\tau)(\partial_j P),\\
\partial_j^\ast(1\otimes P) &= \xi_j P - (\tau\otimes\id)(\partial_j P)
\end{aligned}
\end{equation}
holds for $j=1,\dots,n$ and any $P\in\C\langle X_1,\dots,X_n\rangle$.

\begin{lemma}[Lemma 12 in \cite{Dab-Gamma}]\label{Dab-estimates}
Under the above conditions, we have for each $P\in\C\langle X_1,\dots,X_n\rangle$ that
\begin{equation}\label{Dab-estimates-1}
\begin{aligned}
\|P \xi_j - (\id\otimes\tau)(\partial_j P)\|_2 &\leq \|\xi_j\|_2 \|P\|,\\
\|\xi_j P - (\tau\otimes\id)(\partial_j P)\|_2 &\leq \|\xi_j\|_2 \|P\|
\end{aligned}
\end{equation}
and therefore
\begin{equation}\label{Dab-estimates-2}
\begin{aligned}
\|(\id\otimes\tau)(\partial_j P)\|_2 &\leq 2\|\xi_j\|_2 \|P\|,\\
\|(\tau\otimes\id)(\partial_j P)\|_2 &\leq 2\|\xi_j\|_2 \|P\|.
\end{aligned}
\end{equation}
\end{lemma}

Note that Lemma \ref{Dab-estimates} is actually a slight extension of Lemma 12 in \cite{Dab-Gamma}, since we added in \eqref{Dab-estimates-1} and \eqref{Dab-estimates-2} each time the second named estimates. In fact, they can be easily deduced from the first named estimates by using the more general identity
\begin{equation}\label{duality}
(\tau\otimes\id)(P_1 (\partial_i P_2))^\ast = (\id\otimes\tau)((\partial_i P_2^\ast) P_1^\ast)
\end{equation}
for all $P_1,P_2\in\C\langle X_1,\dots,X_n\rangle$, which can itself easily be checked on monomials.

Moreover, we note that thanks to \eqref{conjugate_formulas}, the inequalities in \eqref{Dab-estimates-1} can be rewritten in the following way:
\begin{equation}\label{conjugate_norms}
\begin{aligned}
\|\partial_j^\ast(P \otimes 1)\|_2 &\leq \|\xi_j\|_2 \|P\|\\
\|\partial_j^\ast(1 \otimes P)\|_2 &\leq \|\xi_j\|_2 \|P\|
\end{aligned}
\end{equation}

\subsection{Proof of Theorem \ref{main2}}

Inspired by the proof of Theorem \ref{main1}, we try to find a certain reduction argument that allows us to lower the degree of $P$. This will need some preparation.
 
\begin{lemma}\label{approximation}
For any $w=w^\ast\in \vN(X_1,\dots,X_n)$, there exists a sequence $(w_k)_{k\in\N}$ of elements in $\C\langle X_1,\dots,X_n\rangle$ such that:
\begin{itemize}
 \item[(i)] $w_k = w_k^\ast$
 \item[(ii)] $\displaystyle{\sup_{k\in\N} \|w_k\| < \infty}$
 \item[(iii)] $\|w_k-w\|_2 \rightarrow 0$ for $k\rightarrow\infty$
\end{itemize}
\end{lemma}

\begin{proof}
First of all, we note that in order to prove the statement of the lemma above, it suffices to prove existence of a sequence $(w_k)_{k\in\N}$ of elements in $\C\langle X_1,\dots,X_n\rangle$, which satisfy only conditions (ii) and (iii). Indeed, if we replace in this case $w_k$ by its real part $\Re(w_k)=\frac{1}{2}(w_k+w_k^\ast)$, conditions (ii) and (iii) are still valid, but we have achieved condition (i) in addition.

For proving existence under these weaker conditions, we may apply Kaplansky's density theorem. This guarantees the existence of a sequence $(w_k)_{k\in\N}$ of elements in $\C\langle X_1,\dots,X_n\rangle$, such that $\|w_k\| \leq \|w\|$ for all $k\in\N$, which particularly implies (ii), and which converges to $w$ in the strong operator topology. It remains to note that, with respect to the weak operator topology, $w_k^\ast w \rightarrow w^\ast w$, $w^\ast w_k \rightarrow w^\ast w$, and $w_k^\ast w_k \rightarrow w^\ast w$ as $k\rightarrow\infty$, such that according to the continuity of $\tau$
\begin{align*}
\|w_k-w\|_2^2 &= \tau((w_k-w)^\ast(w_k-w))\\
              &= \tau(w_k^\ast w_k) - \tau(w_k^\ast w) - \tau(w^\ast w_k) + \tau(w^\ast w)
\end{align*}
tends to $0$ as $k\rightarrow 0$, which shows (iii) and thus concludes the proof.
\end{proof}

A key observation for the absence of zero divisors is contained in the following proposition.

\begin{proposition}\label{reduction-strong}
Let $P\in\C\langle X_1,\dots,X_n\rangle$ be given. For any $u=u^\ast,v=v^\ast\in \vN(X_1,\dots,X_n)$, the following implication holds true:
$$P u = 0 \quad\text{and}\quad P^\ast v = 0 \implies \forall i=1,\dots,n:\ v\otimes 1 (\partial_i P) 1 \otimes u  = 0.$$
\end{proposition}

\begin{proof}
Let $(u_k)_{k\in\N}$ and $(v_k)_{k\in\N}$ be sequences as described in Lemma \ref{approximation} for $u$ and $v$, respectively. Choose arbitrary non-commutative polynomials $Q_1,Q_2\in\C\langle X_1,\dots,X_n\rangle$. Then, for $i=1,\dots,n$ and any $k\in\N$, we can check that
\begin{align*}
\lefteqn{\overbrace{\langle Pu_k, \partial_i^\ast(v_k Q_1\otimes Q_2) \rangle}^{\text{(i)}}}\\
&= \langle \partial_i (Pu_k), v_k Q_1 \otimes Q_2\rangle\\
&= \langle (\partial_i P) 1 \otimes u_k, v_k Q_1 \otimes Q_2\rangle + \langle P \otimes 1 (\partial_i u_k), v_k Q_1 \otimes Q_2\rangle\\
&= \langle (\partial_i P) 1 \otimes (u_k-u), v_k Q_1 \otimes Q_2\rangle + \langle (\partial_i P) 1\otimes u, v_k Q_1 \otimes Q_2\rangle\\
& \qquad + \langle P \otimes 1 (\partial_i u_k), v_k Q_1 \otimes Q_2\rangle\\
&= \langle (\partial_i P) 1 \otimes (u_k-u), v_k Q_1 \otimes Q_2\rangle + \langle (\partial_i P) 1\otimes u, (v_k-v) Q_1 \otimes Q_2\rangle\\
& \qquad + \langle (\partial_i P) 1\otimes u, v Q_1 \otimes Q_2\rangle + \langle P \otimes 1 (\partial_i u_k), v_k Q_1 \otimes Q_2\rangle\\
&= \langle (\partial_i P) 1 \otimes (u_k-u), v_k Q_1 \otimes Q_2\rangle + \langle (\partial_i P) 1\otimes u, (v_k-v) Q_1 \otimes Q_2\rangle\\
& \qquad + \langle v \otimes 1 (\partial_i P) 1\otimes u, Q_1 \otimes Q_2\rangle + \langle (\partial_i u_k) 1 \otimes Q_2^\ast, P^\ast v_k Q_1 \otimes 1\rangle\\
&= \underbrace{\langle (\partial_i P) 1 \otimes (u_k-u), v_k Q_1 \otimes Q_2\rangle}_{\text{(iii)}} + \underbrace{\langle (\partial_i P) 1 \otimes u, (v_k-v) Q_1 \otimes Q_2\rangle}_{\text{(iv)}}\\
& \qquad + \langle v \otimes 1 (\partial_i P) 1 \otimes u, Q_1 \otimes Q_2\rangle + \underbrace{\langle (\id\otimes\tau)((\partial_i u_k) 1 \otimes Q_2^\ast), P^\ast v_k Q_1\rangle}_{\text{(ii)}}.
\end{align*}
Now, we will separately discuss the terms appearing in the above calculation:
\begin{itemize}
 \item[(i)] According to Voiculescu's formula, which we recalled in Lemma \ref{Voi-formula}, we have for all $Y_1,Y_2 \in \C\langle X_1,\dots,X_n\rangle$ that
 \begin{align*}
 \lefteqn{\partial_i^\ast (Y_1 \otimes Y_2)}\\
  &= Y_1 \xi_i Y_2 - m_1(\id\otimes\tau\otimes\id)(\partial_i \otimes \id + \id \otimes \partial_i) Y_1 \otimes Y_2\\
  &= Y_1 \xi_i Y_2 - (\id\otimes\tau)(\partial_i Y_1) Y_2 - Y_1 (\tau\otimes\id)(\partial_i Y_2)\\
  &= \partial_i^\ast(Y_1 \otimes 1) Y_2 - Y_1 (\tau\otimes\id)(\partial_i Y_2)\\
  &= \partial_i^\ast(Y_1 \otimes 1) Y_2 - Y_1 (\id\otimes\tau)(\partial_i Y_2^\ast)^\ast
 \end{align*}
 and thus, by applying the estimates \eqref{Dab-estimates-2} and \eqref{conjugate_norms}, that
 \begin{align*}
 \lefteqn{\|\partial_i^\ast(Y_1 \otimes Y_2)\|_2}\\
 &\leq \|\partial_i^\ast(Y_1 \otimes 1)\|_2 \|Y_2\| + \|Y_1\| \|(\id\otimes\tau)(\partial_i Y_2^\ast)\|_2\\
 &\leq 3 \|\xi_i\|_2 \|Y_1\| \|Y_2\|.
 \end{align*}
 Therefore, we may deduce that
 \begin{align*}
 \lefteqn{|\langle Pu_k, \partial_i^\ast(v_k Q_1\otimes Q_2) \rangle|}\\
 &\leq \|Pu_k\|_2 \|\partial_i^\ast(v_k Q_1 \otimes Q_2)\|_2\\
 &= \|P(u_k-u)\|_2 \|\partial_i^\ast(v_k Q_1 \otimes Q_2)\|_2\\
 &\leq 3 \|\xi_i\|_2 \|P\| \|Q_1\| \|Q_2\| \|v_k\| \|u_k-u\|_2.
 \end{align*}

 \item[(ii)] We apply partial integration in order to obtain
 \begin{align*}
 \lefteqn{(\id\otimes\tau)((\partial_i Y_1) 1 \otimes Y_2)}\\
 &= (\id\otimes\tau)(\partial_i(Y_1 Y_2)) - (\id\otimes\tau)(Y_1 \otimes 1 (\partial_i Y_2))\\
 &= (\id\otimes\tau)(\partial_i(Y_1 Y_2)) - Y_1 (\id\otimes\tau)(\partial_i Y_2)
 \end{align*}
 for arbitrary $Y_1,Y_2\in\C\langle X_1,\dots,X_n\rangle$. From this, we can easily deduce by using \eqref{Dab-estimates-2} that
 \begin{align*}
 \lefteqn{\|(\id\otimes\tau)((\partial_i Y_1) 1 \otimes Y_2)\|_2}\\
 &\leq \|(\id\otimes\tau)(\partial_i(Y_1 Y_2))\|_2 + \|Y_1\| \|(\id\otimes\tau)(\partial_i Y_2)\|_2\\
 &\leq 4 \|\xi_i\|_2 \|Y_1\| \|Y_2\|.
 \end{align*}
 Hence, it follows that
 \begin{align*}
 \lefteqn{|\langle (\id\otimes\tau)((\partial_i u_k) 1 \otimes Q_2^\ast), P^\ast v_k Q_1\rangle|}\\
 &\leq \|(\id\otimes\tau)((\partial_i u_k) 1 \otimes Q_2^\ast)\|_2 \|P^\ast v_k Q_1\|_2\\
 &= \|(\id\otimes\tau)((\partial_i u_k) 1 \otimes Q_2^\ast)\|_2 \|P^\ast (v_k-v) Q_1\|_2\\
 &\leq 4 \|\xi_i\|_2 \|P\| \|Q_1\| \|Q_2\| \|u_k\| \|v_k-v\|_2
 \end{align*}

 \item[(iii)] Note that, according to \eqref{duality}, our calculations from (ii) imply
 $$\|(\tau\otimes\id)(Y_1 \otimes 1(\partial_i Y_2))\|_2 \leq 4 \|\xi_i\|_2 \|Y_1\| \|Y_2\|$$
 for all $Y_1,Y_2\in\C\langle X_1,\dots,X_n\rangle$. This allows us to deduce from
 \begin{align*}
 \lefteqn{\langle (\partial_i P) 1 \otimes (u_k-u), v_k Q_1 \otimes Q_2\rangle}\\
 &= \langle Q_1^\ast v_k \otimes 1 (\partial_i P), 1 \otimes Q_2(u_k-u)\rangle\\
 &= \langle (\tau\otimes\id)(Q_1^\ast v_k \otimes 1 (\partial_i P)), Q_2(u_k-u)\rangle
 \end{align*}
 that
 \begin{align*}
 \lefteqn{|\langle (\partial_i P) 1 \otimes (u_k-u), v_k Q_1 \otimes Q_2\rangle|}\\
 &\leq \|(\tau\otimes\id)(Q_1^\ast v_k \otimes 1 (\partial_i P))\|_2 \|Q_2(u_k-u)\|_2\\
 &\leq 4 \|\xi_i\|_2 \|P\| \|Q_1\| \|Q_2\| \|v_k\| \|u_k-u\|_2.
 \end{align*}

 \item[(iv)] It remains to observe that
 \begin{align*}
 \lefteqn{|\langle (\partial_i P) 1 \otimes u, (v_k-v) Q_1 \otimes Q_2\rangle|}\\
 &\leq \|\partial_i P\|_2 \|Q_1\| \|Q_2\| \|u\| \|v_k-v\|_2.
 \end{align*}
\end{itemize}
In the limit $k\rightarrow \infty$, a combination of all estimates proved in (i) up to (iv) shows that
$$\langle v \otimes 1 (\partial_i P) 1 \otimes u, Q_1 \otimes Q_2\rangle = 0.$$
Since $Q_1,Q_2\in\C\langle X_1,\dots,X_n\rangle$ were arbitrarily chosen, it follows
$$v \otimes 1 (\partial_i P) 1 \otimes u = 0$$
for all $i=1,\dots,n$ as claimed.
\end{proof}

\begin{remark}
Let $P\in\C\langle X_1,\dots,X_n\rangle$ be given and assume that there are $u=u^\ast,v=v^\ast \in \vN(X_1,\dots,X_n)$ such that $Pu=P^\ast v=0$ holds. Then, according to Proposition \ref{reduction-strong}, we know that $v\otimes 1 (\partial_j P) 1 \otimes u = 0$ for any $j=1,\dots,n$. If we replace now $P$ by $P^\ast$, the statement of Proposition \ref{reduction-strong} also gives $u\otimes 1(\partial_jP^\ast)1\otimes v=0$ for $j=1,\dots,n$. But we want to point out that this does not lead to any new information. Indeed, if we take adjoints in the initial statement
$$v\otimes 1(\partial_j P)1\otimes u = 0,$$
we get
$$1\otimes u(\partial_j P)^\ast v \otimes 1 = 0.$$
Then, if we apply the flip $\sigma: M \otimes M \to M \otimes M$, i.e. the $\ast$-homomorphism induced by $\sigma(a_1 \otimes a_2) := a_2 \otimes a_1$, it follows that
$$u\otimes 1\sigma((\partial_j P)^\ast)1\otimes v = 0.$$
An easy calculation on monomials shows $\sigma((\partial_j P)^\ast) = \partial_j P^\ast$, such that the above result reduces exactly to the statement obtained by replacing $P$ with $P^\ast$.
\end{remark}

Before doing the final step, we first want to test in two examples how strong the result in Proposition \ref{reduction-strong} is.

\begin{example}
For the self-adjoint polynomial $P=X_1X_2X_1$, we calculate $\partial_2 P = X_1 \otimes X_1$, such that $Pw=0$ implies according to Proposition \ref{reduction-strong} that $wX_1 \otimes X_1w = 0$ and therefore $X_1w=0$ holds.

Applying now Proposition \ref{reduction-strong} once again with respect to $\partial_1$, we end up with $w\otimes w=0$, such that $w=0$ follows.
\end{example}

\begin{example}
Take $P=X_1X_2+X_2X_1$. We have
\begin{align*}
\partial_1 P &= 1 \otimes X_2 + X_2 \otimes 1,\\
\partial_2 P &= X_1 \otimes 1 + 1 \otimes X_1
\end{align*}
and thus according to Proposition \ref{reduction-strong}
\begin{align*}
(X_2 w)^\ast (X_2 w) &= m_{X_2}( w \otimes 1 (\partial_1 P) 1 \otimes w) = 0,\\
(X_1 w)^\ast (X_1 w) &= m_{X_1}( w \otimes 1 (\partial_2 P) 1 \otimes w) = 0.
\end{align*}
We conclude $X_1w=X_2w=0$, from which we may deduce like above by a second application of Proposition \ref{reduction-strong} that $w=0$.
\end{example}

Although the above examples might give the feeling that Proposition \ref{reduction-strong} is strong enough to allow directly a successive reduction of any polynomial, the needed algebraic manipulations turn out to be obscure in general.

Anyhow, in contrast to Theorem \ref{main2}, any result like this would need a symmetric starting condition. Thus, we will use the following general lemma, which is an easy consequence of the polar decomposition and encodes the additional information that our statement is formulated in a tracial setting. 

\begin{lemma}\label{key-lemma}
Let $x$ be an element of any tracial $W^\ast$-probability space $(M,\tau)$ over some complex Hilbert space $H$. Let $p_{\ker(x)}$ and $p_{\ker(x^\ast)}$ denote the orthogonal projections onto $\ker(x)$ and $\ker(x^\ast)$, respectively.

The projections $p_{\ker(x)}$ and $p_{\ker(x^\ast)}$ belong both to $M$ and satisfy
$$\tau(p_{\ker(x)}) = \tau(p_{\ker(x^\ast)}).$$
Thus, in particular, if $\ker(x)$ is non-zero, then also $\ker(x^\ast)$ is a non-zero subspace of $H$.
\end{lemma}

\begin{proof}
We consider the polar decomposition $x = v(x^\ast x)^{1/2} = (x x^\ast)^{1/2} v$ of $x$, where $v\in M$ is a partial isometry mapping $\overline{\ran(x^\ast)}$ to $\overline{\ran(x)}$, such that
$$v^\ast v = p_{\overline{\ran(x^\ast)}} \qquad\text{and}\qquad v v^\ast = p_{\overline{\ran(x)}}.$$
Hence, it follows that
$$1-v^\ast v = p_{\ran(x^\ast)^\bot} = p_{\ker(x)} \qquad\text{and}\qquad 1 - v v^\ast = p_{\ran(x)^\bot} = p_{\ker(x^\ast)},$$
from which we may deduce by traciality of $\tau$ that indeed
$$\tau(p_{\ker(x)}) = \tau(1-v^\ast v) = \tau(1 - v v^\ast) = \tau(p_{\ker(x^\ast)}).$$
This concludes the proof.
\end{proof}

Combining Lemma \ref{key-lemma} with Proposition \ref{reduction-strong} will provide us with the desired reduction argument. Before giving the precise statement, let us introduce some notation. If $p\in M$ is any projection, we define a linear mapping
$$\Delta_{p,j}:\ \C\langle X_1,\dots,X_n\rangle \rightarrow \C\langle X_1,\dots,X_n\rangle$$
for $j=1,\dots,n$ by 
$$\Delta_{p,j} P := (\tau \otimes \id)((p\otimes 1) \partial_j P)$$ 
for any $P\in \C\langle X_1,\dots,X_n\rangle$.

\begin{corollary}\label{reduction-final}
Let $P\in \C\langle X_1,\dots,X_n\rangle$ and $w=w^\ast\in \vN(X_1,\dots,X_n)$ be given, such that $Pw=0$ holds true. If $w\neq0$, then there exists a projection $0\neq p\in \vN(X_1,\dots,X_n)$ such that $(\Delta_{p,j} P) w = 0$.
\end{corollary}

\begin{proof}
Since $Pw=0$ and $w\neq0$, we see that $\{0\} \neq \ran(w) \subseteq \ker(P)$, such that we also must have $\ker(P^\ast) \neq \{0\}$ according to Lemma \ref{key-lemma}. The projection $p := p_{\ker(P^\ast)} \in\vN(X_1,\dots,X_n)$ thus satisfies $p\neq0$ and $P^\ast p = 0$. Proposition \ref{reduction-strong} tells us that $(p\otimes 1) (\partial_j P) (1 \otimes w) = 0$ for $j=1,\dots,n$ holds true. Hence, we get that
$$(\Delta_{p,j} P) w = (\tau\otimes\id)((p\otimes 1) \partial_j P) w = (\tau\otimes\id)((p\otimes 1) (\partial_j P) (1 \otimes w)) = 0,$$
which concludes the proof.
\end{proof}

Now, we are prepared to finish the proof of Theorem \ref{main2}. It suffices to show that, if $P\in\C\langle X_1,\dots,X_n\rangle$ and $w=w^\ast\in \vN(X_1,\dots,X_n)$ with $w\neq0$ are given such that $Pw=0$, then $P=0$ follows. For seeing this, write
$$P = a_0 + \sum_{k=1}^d \sum_{i_1,\dots,i_k = 1}^n a_{i_1,\dots,i_k} X_{i_1} \dots X_{i_k}.$$
Assume that the total degree $d$ of $P$ satisfies $d\geq1$. We choose any summand of highest degree
$$a_{i_1,\dots,i_d} X_{i_1} \dots X_{i_d}$$
of $P$, which is non-zero. Iterating Corollary \ref{reduction-final}, we see that there are non-zero projections $p_1,\dots,p_d \in \vN(X_1,\dots,X_n)$ such that
$$(\Delta_{p_d,i_d} \dots \Delta_{p_1,i_1} P) w = 0.$$
But since we can easily check that
$$\Delta_{p_d,i_d} \dots \Delta_{p_1,i_1} P = \tau(p_d) \dots \tau(p_1) a_{i_1,\dots,i_d},$$
this leads us to $a_{i_1,\dots,i_d} = 0$, which contradicts our assumption. Thus, $P$ must be constant, and since $w\neq0$, we end up with $P=0$. This concludes the proof of Theorem \ref{main2}.

\bibliographystyle{amsalpha}
\bibliography{zero_divisors}

\end{document}